\documentclass{amsart}
\usepackage{mathtools}
\usepackage{amsfonts}
\usepackage{amsmath}
\usepackage{amssymb}
\usepackage{amsthm}
\usepackage{breakurl}
\usepackage{tikz}
\usepackage{hhline}
\usepackage{enumerate}
\usepackage{mathrsfs}
\usepackage{tabularx}
\usepackage{booktabs}
\usepackage{colonequals}
\usepackage[hyphens,spaces,obeyspaces]{url}
\usepackage{hyperref}
\usepackage[boxed, section]{algorithm}
\usepackage{algpseudocode}
\usepackage[color = lime]{todonotes}
\usepackage{graphicx}
\usepackage{pgfplots}

\pgfplotsset{compat=1.14}
\theoremstyle{plain}\newtheorem{thm}{Theorem}[section]
\theoremstyle{plain}\newtheorem{defn}[thm]{Definition}
\theoremstyle{definition}\newtheorem{eg}[thm]{Example}
\theoremstyle{plain}
\theoremstyle{plain}
\theoremstyle{definition}
\theoremstyle{definition}
\theoremstyle{definition}
\theoremstyle{definition}
\theoremstyle{plain}\newtheorem{alg}[thm]{Algorithm}
\theoremstyle{plain}\newtheorem{rmk}[thm]{Remark}
\theoremstyle{plain}\newtheorem{lem}[thm]{Lemma}
\theoremstyle{plain}
\theoremstyle{plain}

\newcommand{\pp}{\mathbb{P}}

\newcommand{\qq}{\mathbb{Q}}
\newcommand{\zz}{\mathbb{Z}}
\newcommand{\ff}{\mathbb{F}}

\newcommand{\nn}{\mathbb{N}}
\newcommand{\FF}{\mathbb{F}}

\newcommand{\fp}{\mathfrak{p}}

\newcommand{\comment}[1]{}

\newcommand{\Div}{\operatorname{div}}

\newcommand{\Gal}{\operatorname{Gal}}

\newcommand{\Aut}{\operatorname{Aut}}
\newcommand{\SL}{\operatorname{SL}}

\newcommand{\ord}{\operatorname{ord}}

\newcommand{\Sym}{\operatorname{Sym}}

\newcommand{\Van}{\text{Van}}
\newcommand{\p}{^\prime}

\newcommand{\tors}{\operatorname{tors}}

\renewcommand{\d}[1]{\ensuremath{\operatorname{d}\!{#1}}}

\newcommand{\defi}[1]{\textsf{#1}}

\title{Chabauty--Coleman computations on rank 1 Picard curves}

\author{Sachi Hashimoto and Travis Morrison}
\date{}
\address{Sachi Hashimoto, Department of Mathematics and Statistics, Boston University, 111 Cummington Mall, Boston, MA 02215, USA}
\thanks{SH was supported by a Clare Boothe Luce Fellowship (Henry Luce Foundation) and NSF Graduate Research Fellowship grant number DGE-1840990. TM was supported by funding from the Natural Sciences and Engineering Research Council of Canada, the Canada First Research
Excellence Fund, CryptoWorks21, Public Works and Government Services
Canada, and the Royal Bank of Canada.}
\email{svh@bu.edu}

\address{Travis Morrison, Mathematics Department, Virginia Tech, 554 McBryde Hall,
225 Stanger Street,
Blacksburg, VA 24061}
\email{tmo@vt.edu}

\begin{document}

\begin{abstract}
We provably compute the full set of rational points on 1403 Picard curves defined over $\qq$ with Jacobians of Mordell--Weil rank $1$ using the Chabauty--Coleman method. To carry out this computation, we extend Magma code of Balakrishnan and Tuitman for Coleman integration. The new code computes $p$-adic (Coleman) integrals on curves to points defined over number fields where the prime $p$ splits completely and implements effective Chabauty for curves whose Jacobians have infinite order points that are not the image of a rational point under the Abel--Jacobi map. We discuss several interesting examples of curves where the Chabauty--Coleman set contains points defined over number fields.
\end{abstract}

\maketitle

\section{Introduction}

Let $X/\qq$ be a smooth projective curve of genus at least $2$. Faltings's theorem \cite{Faltings}, proved in 1983, establishes that the set of rational points $X(\qq)$ of $X$ is finite, but falls short of giving an explicit method for computing this set. Our goal in this paper is to compute this set explicitly for certain curves of genus $3$.

Let $p$ be a prime of good reduction for $X$. We focus on curves amenable to Chabauty--Coleman calculations, which use $p$-adic (Coleman) integrals to carve out a finite set of $p$-adic points containing the rational points. The Chabauty--Coleman method is one attempt at an effective Faltings's theorem. The method dates back to work of Chabauty in 1941 \cite{Chabauty}, who showed finiteness of $X(\qq)$ for certain curves $X$ using the embedding of $X(\qq)$ into the $p$-adic closure of $J(\qq)$, where $J$ is the Jacobian variety of $X$. In the 1980s, Coleman \cite{EffectiveChab,Torsion} made Chabauty's approach effective with the development of the Coleman integral. With this computational tool, he was able to bound the number of rational points on some curves.

For a given curve $X$, the Chabauty--Coleman method requires that the Mordell--Weil rank of its Jacobian be computable and smaller than the genus of $X$. In this paper we apply the Chabauty--Coleman method to genus 3 non-hyperelliptic curves. While there exist techniques to compute the rank of Jacobians of general smooth plane quartics \cite{BruinPoonenStoll}, practical implementations of rank-bounding techniques only exist for superelliptic curves, based on work of Poonen and Schaefer \cite{PoonenSchaefer}. A \defi{Picard curve} is a smooth plane quartic $X/\qq$ that admits a Galois cover of degree $3$ of $\pp^1$. Every Picard curve has an affine equation of the form $y^3 = f(x)$ where $f(x) \in \qq[x]$ is degree $4$ and without multiple roots in $\overline{\qq}$. Therefore we focus on Picard curves. While for some particular genus $3$ curves $X$ of Mordell--Weil rank $2$ it is possible to determine the rational points using the Chabauty--Coleman method, and one can always extend the method using the Mordell--Weil sieve \cite{SiksekMordell,BruinStollMW}, which combines $p$-adic information for multiple primes $p$, when the rank is one less than the genus we expect the Chabauty--Coleman method alone to fail to determine $X(\qq)$ for most $X$. So, in particular, we focus on the rank $1$ case, where we can work with at least two linearly independent functionals on $J(\qq_p)$.

Recently, Balakrishnan and Tuitman \cite{balatuitman} have made Coleman integration explicit and practical, opening up the ability to compute rational points on entire databases of curves satisfying Chabauty's rank hypothesis.

Sutherland has computed a database of genus 3 curves of bounded discriminant \cite{SutherlandNonHyp, g3database}, providing the necessary data for large-scale computational experiments in genus 3. Balakrishnan--Bianchi--Cantoral-Farf\'an--{\c{C}}iperiani--Etropolski \cite{CCExp} computed rational points on a database of genus 3 hyperelliptic curves with Mordell--Weil rank 1 using Chabauty--Coleman methods and de Frutos-Fern\'andez--Hashimoto \cite{MariaSachi} computed rational points on genus 3 hyperelliptic curves with Mordell--Weil rank 0. This work is the first instance of such computations being carried out on non-hyperelliptic curves.

The Chabauty--Coleman method cuts out a finite set of $p$-adic points $X(\qq_p)_1$ on a curve $X$ that contains the rational points $X(\qq)$; sometimes this set is strictly larger than the set of rational points.

For every Picard curve $X$ we consider, we classify the set of extra points $X(\qq_p)_1 \setminus X(\qq)$ in the Chabauty--Coleman set. We  show that each extra point $P$ is \defi{defined over a number field}, that is, we can find a number field $K$ where $p$ is divisible by a prime $\mathfrak{p}$ of degree one dividing $p$, and $P \in X(K) \cap X(K_\mathfrak{p})$. Let $\infty \in X(\qq)$ be the unique rational point at infinity. These extra points fall into one of three categories:

\begin{itemize}
\item $P \in X(K)$ such that $[P-\infty]$ is in $J(K)_{\tors}$;
\item $P \in X(K)$ such that there exists $n \in \zz$ so that $n[P-\infty]$ is infinite order in $J(\qq)$;
\item $P \in X(K)$ that appear due to the existence of extra automorphisms of $X_{\overline{\qq}}$.
\end{itemize}
We exhibit examples of all three possibilities in Section \ref{examples}.

Our Magma code extending the functionality of Balakrishnan and Tuitman's Coleman integration code as well as our code for computing $X(\qq_p)_1$ on the database of Picard curves is available at ~\cite{TravisSachiextensions, TravisSachiexperiments}.

\section{The Chabauty--Coleman method}

\subsection{Chabauty--Coleman}
In this section we introduce the Coleman integral as a $p$-adic tool for computing the rational points on a curve. We follow the exposition of the construction of the Coleman integral in~\cite{Wetherell} and \cite{McCallumPoonen}. Let $X$ be a smooth projective curve of genus $g>1$ over $\qq$ with Jacobian $J$. Let $B$ be a nonzero rational divisor on $X$ inducing an Abel--Jacobi map
\begin{align*}
\iota\colon &X \rightarrow J  \\
&Q \mapsto [\deg(B)Q - B]
\end{align*}
 (for example we can always take a canonical divisor or we can take a rational point if one is known). Let $p$ be a prime of good reduction for $X$. We introduce two ways to construct the $p$-adic Coleman integral of a regular one-form on $X_{\qq_p}$: one by pulling back an integral on $J_{\qq_p}$ thought of as a $p$-adic Lie group and the other directly on $X_{\qq_p}$ as a rigid analytic space. We first describe the $p$-adic Lie group approach and do not address the equivalence here; Coleman discusses the details in \cite{Torsion}.

Let $\omega \in H^0(J_{\qq_p}, \Omega^1)$.
Since $J(\qq_p)$ is a $p$-adic Lie group, the differential $\omega$ is translation-invariant, and we can define an antiderivative homomorphism $\int_0: J(\qq_p)\to \qq_p$. This homomorphism sends a point $Q \in J(\qq_p)$ to $\int_0^Q \omega$. For any open subgroup $U \subset J(\qq_p)$, the antiderivative of $\omega$ can be computed by expanding $\omega$ in local coordinates at $Q \in U$ and taking formal antiderivative in the power series ring; this yields a convergent $p$-adic power series on small enough open subgroups.

 The pullback map
\[
 \iota^* \colon H^0(J_{\qq_p}, \Omega^1) \to H^0(X_{\qq_p}, \Omega^1)
 \]
  induced by $\iota$ is an isomorphism, independent of the choice of $B$~\cite[Proposition~2.2]{MilneJacobian}. This gives us a way to think of the integral on the curve $X_{\qq_p}$: for any $\omega \in H^0(X_{\qq_p}, \Omega^1)$ and any $P, Q\in X(\qq_p)$, we can define $\int_P^{Q}\omega$ by $\int_0^{[Q-P]} (\iota^*)^{-1}\omega$. Conversely, we can write an integral on $J_{\qq_p}$ in terms of integrals on $X_{\qq_p}$: for any $D \in J(\qq_p)$ and $\omega \in H^0(J_{\qq_p}, \Omega^1)$, we can write $D = \sum_i [Q_i- P_i]$ where $P_i \in X(L)$ for $L / \qq_p$ a finite extension. Then \[\int_0^D  \omega = \sum_i \int_{P_i}^{Q_i} \iota^*\omega.\] Note $\int_0^D \omega =0$ when $D\in J(\qq_p)_{\tors}$ because the antiderivative is a homomorphism.

 Coleman gives a definition of the integral $\int_P^Q \eta$ on the rigid analytic space $X_{\qq_p}$; to do so, we work over a wide open space $V \subset X_{\qq_p}$ constructed by deleting a finite number of closed disks from $X_{\qq_p}$ of radius less than 1. He defines an integral for $P, Q \in V$ and $\eta$ any differential $\omega$ of the second kind on $V$ \cite{Torsion} (a differential $\omega$ is \defi{of the second kind} if it is everywhere meromorphic and the residue at every pole is zero). Coleman shows his definition agrees with the one induced by the $p$-adic Lie group structure of $J_{\qq_p}$ when $\eta \in H^0(X_{\qq_p}, \Omega^1)$. Fix an algebraic closure $\overline{\qq}_p$ of $\qq_p$. For all points $P, Q, R \in V(\overline{\qq}_p)$ and one-forms $\omega, \eta$ of the second kind on $X_{\qq_p}$, the Coleman integral enjoys the following properties:

  \begin{thm}(Coleman)
  \label{thm:properties}
\begin{enumerate}
	\item Linearity: $\int_P^{Q}(\alpha\omega +\beta\eta)=\alpha\int_P^{Q}\omega + \beta\int_P^{Q}\eta$, for all $\alpha, \beta \in \qq_p$.
	\item Additivity: $\int_{P}^Q \omega = \int_P^{R} \omega + \int_{R}^Q \omega$.
	\item Fundamental theorem of calculus:
	 $\int_P^Q \d f = f(Q) - f(P)$ for $f$ a rigid analytic function on $V$.
	\item Change of variables: if $V\p \subset X\p$ is another wide open subspace of a rigid analytic space $X\p$, for any rigid analytic map $\phi\colon V\to V\p$, we have $\int_{\phi(P)}^{\phi(Q)}\omega = \int_P^{Q}\phi^*(\omega)$.
	\item
	\label{jacobianprop} For any divisor $D \coloneqq \sum_i Q_i- P_i$ of degree zero on $X$ defined over $\overline{\qq}_p$, and for $\omega$ a regular one-form, then $\int_D \omega \colonequals \sum_i \int_{P_i}^{Q_i} \omega $ is well-defined, and $\int_D \omega = 0$ when $D$ is principal.
	\item
	\label{galoisprop} Galois equivariance: if $\sigma \in \Gal(\overline{\qq}_p/\qq_p)$, then $\left(\int_P^Q \omega\right)^\sigma = \int_{\sigma(P)}^{\sigma(Q)} \omega^\sigma.$
	\label{item:sumD}
\end{enumerate}

\end{thm}

\begin{rmk}
Let $Q_i, P_i \in \overline{\qq}_p$ such that $D \colonequals \sum_{i} Q_i - P_i $ is a divisor of degree zero on $X$ defined over $\qq_p$. Let $\omega \in H^0(X_{\qq_p}, \Omega^1)$. By (\ref{jacobianprop}) and (\ref{galoisprop}) of Theorem \ref{thm:properties}, the integral \[\int_{D} \omega = \sum_i \int_{P_i}^{Q_i} \omega \]
is well-defined, only depends on $[D]$, and belongs to $\qq_p$, even though the individual integrals in the sum $\int_{P_i}^{Q_i} \omega$ may not belong to $\qq_p$.

\end{rmk}

When the rank $r$ of $J(\qq)$ is less than the genus $g$ of $X$, the space of \defi{vanishing differentials}
\[
\Van(X(\qq))\coloneqq \left \{\omega \in H^0(X_{\qq_p}, \Omega^1)\colon \int_D  \omega = 0 \text{ for all } D \in J(\qq)\right\}
\] is at least $(g -r)$-dimensional, and thus nonzero. In particular, when $r = 1$, the space of vanishing differentials is exactly $(g-1)$-dimensional: for $D \in J(\qq)$ the integrals $\int_0^D\omega_i$ on a basis of regular differentials $i= 1, \dots, g$ are zero if and only if $D$ is torsion \cite[III, \S7.6, Proposition~12]{Bourbaki}. Since we are guaranteed some $D \in J(\qq)$ which is non-torsion, $\Van(X(\qq))$ cannot be all of $H^0(X_{\qq_p}, \Omega^1)$.

\begin{rmk}

We will think of the $p$-adic integral of $\omega$ in $\Van(X(\qq))$ as the Coleman integral $\int_D \omega \colonequals \sum_i \int_{P_i}^{Q_i} \omega $ taking place directly on the curve $X_{\qq_p}$ as a rigid analytic space (as in Theorem \ref{thm:properties} part \ref{item:sumD}). While we could consider both the $p$-adic Lie group integral $\int_0^D (\iota^*)^{-1}\omega$ on the Jacobian and the Coleman integral when $\omega \in H^0(X_{\qq_p}, \Omega^1)$, in the process of explicitly computing the Coleman integral between two points in different residue disks, we compute integrals of differentials of the second kind that are not everywhere regular (see Section \ref{explicitint}). Therefore throughout the paper we consider all $p$-adic integrals to take place on the curve as a rigid analytic space.

\end{rmk}

Coleman \cite{EffectiveChab} proves a bound on the number of rational points $\# X(\qq) \leq \#X(\ff_p) + 2g-2$ by bounding the number of zeros of a nonzero differential in $\Van(X(\qq))$, assuming $p > 2g$.

Coleman's proof suggests an algorithm for computing $X(\qq)$ by computing the set
\[
X(\qq_p)_1 \coloneqq \left \{Q \in X(\qq_p): \int_D^{\deg(D)Q} \omega = 0\text{ for all } \omega \in \Van(X(\qq))\right \}.
\]
Since the Coleman integral is locally analytic, it can be expressed as the integral of a convergent $p$-adic power series on each residue disk. Computing $X(\qq_p)_1$ yields finitely many solutions in each disk, and so finitely many in total. Therefore $X(\qq)\subseteq X(\qq_p)_1$ is finite.

\subsection{Explicit Coleman integration after Balakrishnan and Tuitman}
\label{explicitint}

To carry out the Chabauty--Coleman method in practice, we need an explicit method for computing the set $X(\qq_p)_1$, which is defined as the vanishing locus of some $p$-adic integrals. In this section, we give an overview of a practical algorithm to compute $p$-adic integrals on curves, following the method of Balakrishnan and Tuitman \cite{balatuitman}. For a more detailed description of the algorithm, which applies in more generality than described here, we refer the reader to their paper, as well as the zeta function algorithm of Tuitman described in \cite{Tuitman1,Tuitman2} which develops the reduction in cohomology, extending an earlier algorithm for hyperelliptic curves due to Kedlaya \cite{Kedlaya,KedlayaErrata}.

Let $X/\qq$ be a smooth projective geometrically integral curve with (possibly singular) affine plane model given by the equation $Q(x,y)=0$ with $Q(x,y) \in \zz[x,y]$ and let $p$ be a prime of good reduction. Tuitman's algorithm exploits the existence of a low-degree map $x:X \to \pp^1$ of degree $d_x$. After removing the ramification locus of this map, the algorithm chooses a lift of Frobenius which sends $x \mapsto x^p$ and Hensel lifts $y$. Then, we compute the action of Frobenius on a certain basis of differentials whose classes generate $H^1_{rig}(X_{\qq_p})$, and reduce using Lauder's fibration method which is given by explicit and fast linear algebra.

To specify a point $P$ on $X$, Balakrishnan and Tuitman's algorithm uses the following data:

\begin{itemize}

\item $\tt{ P `x}$ the $x$-coordinate of $P$ (or $0$ if $x$ is infinite).

\item $\tt{ P` b}$ the values of $b^0= [b^0_0, \dots, b^0_{d_x-1}]$ that form an integral basis for the function field $\qq(X)/ \qq[x]$ at $x$, ($ b^\infty = [b^\infty_0, \dots, b^\infty_{d_x-1}]$ an integral basis for $\qq(X)/ \qq[1/x]$ if $P$ is infinite).

\item $\tt{ P`inf }$ a boolean value specifying whether $P$ is infinite.

\end{itemize}

We must impose some additional assumptions on $Q(x,y)$, $p$, $b^0$, and $b^\infty$, which are listed in \cite[Assumption~2.6]{balatuitman}, in order to compute $p$-adic Coleman integrals on $Q(x,y)=0$ using Balakrishnan and Tuitman's algorithm. These hypotheses can be verified directly for our case of Picard curves from the computation of $b^0$ and $b^\infty$ provided here.

In the case where $X$ is a monic Picard curve $y^3=f(x)$, the map $x:X\to \pp^1$ is projection to the $x$-coordinate, $b^0 =[1,y,y^2] $, $b^\infty=[y, y/x^2, y/x^3]$. The regular differentials $\omega_1 = y\d x/f, \omega_2 = xy\d x/f$ and $\omega_3 = y^2 \d x/f$ form a basis of $H^0(X_{\qq_p},\Omega^1)$. We extend this basis by $\omega_4 = x^3 y \d x / f, \omega_5=xy^2 \d x / f, \omega_6 = x^2 y^2 \d x / f$ to obtain classes $[\omega_1], \dots, [\omega_6]$ that generate $H^1_{rig}(X_{\qq_p})$.

The $p$-adic points $X(\qq_p)$ are a disjoint union of \defi{residue disks}, which are the preimages under reduction modulo $p$ of points in $X(\ff_p)$. After Balakrishnan and Tuitman, we classify the residue disks based on their ramification with respect to the map $x: X\to \pp^1.$ Let $\Delta_y(x)$ be the discriminant of $Q$ thought of as a polynomial in $y$, and $r(x)$ be the squarefree polynomial with the same zeros as $\Delta_y(x)$ given by $r(x) = \Delta_y(x)/\gcd(\Delta_y(x), \frac{d \Delta_y(x) }{d x})$. Then our residue disks partition as follows:

\begin{itemize}

\item \defi{Bad residue disks}  are residue disks that contain a point whose $x$-coordinate is a zero of $r(x)$ or is infinite. Such points are called \defi{very bad points}. If the very bad point of a disk has infinite $x$-coordinate, the disk is also an \defi{infinite residue disk}, and the point is also a \defi{very infinite point}. Otherwise the disk is a \defi{finite residue disk}.
\item \defi{Good residue disks} comprise the remaining residue disks. All good residue disks are finite.

\end{itemize}

Note that each residue disk contains at most one very bad point $P=(x,y)$ that will be defined over an unramified extension of $\qq_p$ \cite[Remark~2.9]{balatuitman}. We denote by $e_P$ the ramification index of the map by $x$ at a very bad point $P$. If $Q(x,y) = y^3 - f(x)$ is a Picard curve, then $r(x) = -27f(x)$, the very bad points are the ramification points (where $y=0$), the point at infinity $\infty$ is $[0:1:0]$, the residue disks containing ramification points are the finite bad disks, and all very bad points have ramification index $3$.

There are two kinds of Coleman integrals on $X$: integrals between two points $P$ and $Q$ in $X (\qq_p)$ in the same residue disk, called \defi{tiny Coleman integrals} and integrals between two points that are not in the same residue disk. Tiny integrals are computed using uniformizing parameters. A uniformizing parameter gives an isomorphism from a residue disk to $p \zz_p$ and transforms a Coleman integral into a power series integral.

\begin{alg} Compute a tiny integral $\int_P^Q \omega$ and uniformizing parameter $t$
\label{tinyintegrals}

Input: $P$ and $Q \in X(\qq_p)$ in the same residue disk, and $\omega$ a differential of the second kind on $X$ such that $\omega$ does not have a pole in the residue disk of $P$ and $Q$.

Output: $\int_P^Q \omega$ and $t$ a uniformizing parameter in the residue disk of $P$ and $Q$

\begin{enumerate}

\item If $P, Q$ are in a bad disk, find $P\p$, the very bad point in the disk, and break up the integral as $\int_P^{P\p} \omega + \int_{P\p}^Q \omega$. Otherwise proceed with $P = P\p$.

\item  Compute a uniformizing parameter $t$ at $P\p$: in a good disk, $t$ is given by $x- x(P\p)$ ; at a very bad point $t=b_i^0-b_i^0(P\p)$ for some $i$ (replace with $b_i^\infty$ for very infinite $P\p$).

\item Compute $x$ and the $b^0$ vector (or $b^\infty$) as a function of $t$ to desired precision, and using these expansions compute $\omega$ as a power series in $t$.

\item Integrate \[ \int_P^Q \omega = \int_{t(P)}^{t(Q)} \omega(t)\] as a power series.

\end{enumerate}

\end{alg}

For more details on the computations of the tiny integral and the uniformizing parameters see \cite[Algorithm~3.7, Proposition~3.2]{balatuitman}.

For the second kind of Coleman integral, between points in different residue disks, we turn to the full machinery of Tuitman's zeta function algorithm to compute the action of Frobenius on cohomology. Let $P$ and $Q$ in $X(\qq_p)$ be points in different good residue disks. We use Tuitman's lift $\phi$ of the Frobenius morphism on $X_{\ff_p}$. Let $\omega_i$ be Tuitman's basis differentials such that $[\omega_i]$ generate $H^1_{rig}(X_{\qq_p})$ for $i = 1, \dots, 2g$. By applying Frobenius to $P$ and $Q$, we obtain\[\int_{P}^Q \omega_i = \int_P ^{\phi(P)} \omega_i + \int_{\phi(P)}^{\phi(Q)} \omega_i + \int _{\phi(Q)}^{Q} \omega_i =\int_P ^{\phi(P)} \omega_i + \int_{P}^{Q} \phi^*\omega_i + \int _{\phi(Q)}^{Q} \omega_i  \] where $P$ and $ \phi(P)$ are in the same residue disk, so define a tiny integral, and similarly for $Q$ and $\phi(Q)$.

So, we expressed the unknown integral as a linear combination of two computable quantities and the integral of $\phi^* \omega_i$. This motivates computing $\phi^*\omega_i$; Balakrishnan and Tuitman compute the reduction in cohomology of $\phi^* \omega_i$ as a sum of basis elements $\sum_j M_{ij} \omega_j$ plus an exact differential $\d f_i$. We compute the reductions of all $\phi^* \omega_i$ and consider the linear system of equations to compute the Coleman integral. For more details see \cite[Algorithm~3.7, Algorithm~3.12]{balatuitman} as well as \cite[Section~4.1]{Tuitman2} which discusses how to construct the $2g$ basis differentials whose classes generate $H^1_{rig}(X_{\qq_p})$. If $P$ lies in a bad disk, some $f_i$ might not converge at $P$, but near the boundary of the disk of $P$ they do; we compute the integral $\int_P^Q \omega$ by taking $S \in X(\qq_p(p^{1/e}))$ on the boundary for $e$ large enough and breaking up the integral as $\int_P^S \omega + \int_S^Q\omega$.

In practice, to compute $X(\qq_p)_1$, we integrate to arbitrary endpoints $\int_b^z \omega$; to do this we iterate over residue disks, breaking up the integral by fixing $P$ in each residue disk and integrating $\int_b^P \omega + \int_P^z \omega$ where $\int_b^P \omega $ sets the constant of integration and $\int_P^z \omega$ is a tiny integral which can be expanded as a $p$-adic power series to arbitrary precision and whose zeros can be found using analytic techniques. Balakrishnan and Tuitman implement the Coleman integral as well as the computation of the Chabauty--Coleman set $X(\qq_p)_1$ in Magma \cite{BalaTuitgit} under the assumption that we can find rational points $P_1, P_2, \dots, P_r \in X(\qq)$ such that the classes $[P_i - b]$ are infinite order and linearly independent in $J(\qq)$, where $r$ is the rank of $J(\qq)$. However, their implementation does not check the precision (for example, when solving for roots of $p$-adic power series the code does not check whether all of the roots obtained are single roots) and needs a few modifications to produce provably correct results. Also, their function for solving for zeros of a power series in a disk may not return all possible roots in $\zz_p$ for some power series due to a Hensel lifting problem.

\section{The algorithm}

In this section, we describe the algorithm we implemented to compute $X(\qq)_1$ for each $X$ a rank 1 Picard curve in our database. These curves come from Sutherland's database of genus 3 curves of small discriminant \cite{SutherlandNonHyp,g3database} and have discriminant $\Delta$ bounded by $10^{12}$. See \cite{SutherlandNonHyp} for a definition of the discriminant of a smooth plane curve; in particular, $p | \Delta$ if and only if $p$ is a prime of bad reduction. Each curve in the database is given by the data of an affine model $y^3 = f(x)$ with $f(x)$ degree four, along with its discriminant $\Delta$.

We began by filtering the database for those curves whose Jacobians have Mordell--Weil rank 1. We ran the Magma function $\tt RankBounds(f,3:ReturnGenerators:=true)$ which computes the rank based on work of Poonen and Schaefer \cite{PoonenSchaefer} and Creutz \cite{Creutz} and is implemented in Magma by Creutz. It gives a lower and upper bound for the rank of the Jacobian of the cyclic cover of $\pp^1$ defined by $y^3 = f(x)$ along with a list of polynomials $g(T)$ which define divisors $\Div(g)$ on $\pp^1$ that lift to rational divisors on $X$. The image of such a divisor in the Jacobian is either a $3$-torsion point or a point of infinite order.

If $\tt RankBounds$ returns upper and lower bounds that prove $J$ is rank 1, we proceed. Otherwise, we discard the curve and move on to the next one.\footnote{For practical purposes, we terminated RankBounds if it ran for 120 seconds without returning an answer. We discarded $520$ curves from the total 5335 curves in the database provided by Sutherland either from terminating RankBounds or due to a ``Runtime error''.} For $J$ rank 1, we consider the non-torsion generators $g(T)$; if there is some non-torsion $g(T)$ of degree one, then we have a rational point $P \in X(\qq)$ such that $[P - \infty]$ is infinite order in $J(\qq)$. Otherwise, we are left with $g(T)$ of higher degree, which correspond to rational divisors that are supported on Galois conjugate points defined over a number field $K$. To handle this second case, which comprises 544 of the 1403 Picard curves in our rank 1 database, we extended functionality in Balakrishnan and Tuitman's code to handle rational divisors on curves when running effective Chabauty computations. This is done by extending the Coleman integral functionality to compute integrals to points on curves defined over number fields when $p$ is totally split in the field.

\subsection{With a rational point whose class in the Jacobian is of infinite order}

\begin{alg}Computing $X(\qq_p)_1$ with $P \in X(\qq)$ such that $\text{[}P - \infty\text{]}$ is infinite order
\label{RatPoint} \\
\noindent Input: a Picard curve $X$ with affine equation $y^3= f(x)$ such that $f$ is monic, with discriminant $\Delta$, and Jacobian $J$ such that $J(\qq)$ is rank 1, along with a point $P\in X(\qq)$ such that $[P-\infty]\in J(\qq)$ has infinite order, and parameters $(N,e)$, where $p^N$ is the $p$-adic precision and $\qq_p(p^{1/e})$ is the totally ramified extension over which we work in the bad disks (see Section~\ref{explicitint}).\footnote{For some heuristics for choosing $e$ given a fixed $N$ see Appendix \ref{eheuristics}.} \\
\noindent Output: $S\subseteq X(\qq)$ and $T\subseteq X(\qq_p)_1 \setminus S$ or ``Failure''.
 If the algorithm does not return ``Failure'', and $T=\emptyset$, then $S=X(\qq)$.

\begin{enumerate}

\item Find the first\footnote{Due to a minor bug in the code of Balakrishnan and Tuitman for computing local coordinates at very infinite points, which in certain cases yields a Runtime Error instead of computing the coordinates, we sometimes choose the next largest possible prime.} prime $p>3$ such that $p \nmid \Delta$. (We discard $p=3$ because $3$ always divides the discriminant of a Picard curve \cite[Theorem~1.1]{PicardConductor}. We avoid $p=2$ for simplicity when Hensel lifting.)

\item\label{step: pointsearch} Compute a list of points $S\subseteq X(\qq)$ by searching in a box up to naive height $1000$.

\item Compute a basis $v_1, v_2$ of vanishing differentials $\Van(X(\qq))$ by computing a basis for the kernel of the matrix of integrals $[\int_{\infty}^P \omega_i]$ for $i = 1, 2, 3$ regular basis differentials. (We compute $v_1, v_2$ to finite precision: let $M$ be the matrix of Frobenius and $I$ the identity matrix. Then $v_1, v_2$ are computed to precision $N - \ord_p(\det(M-I)) -\delta$ as described in \cite[Section~4]{balatuitman} where $\delta$ is a precision loss defined in \cite[Definition~4.4]{Tuitman2}. If $  N-\ord_p(\det(M-I)) -\delta <0$, return ``Failure''.)

\item  \label{Integrate} Solve for $X(\qq_p)_1 = \{ Q \in X(\qq_p) | \int_\infty^Q v_i = 0, i= 1,2 \}$ using Algorithm \ref{SolveTiny}. If Algorithm \ref{SolveTiny} returns ``Failure'' (due to a bad choice of parameters $N$ and $e$), return ``Failure''.

\item Let $T \coloneqq X(\qq_p)_1\setminus S$.

\item  Run the Magma function $\tt PowerRelation$ to see if we can recognize the points of $T$ as defined over a number field. If any point $Q$ of $T$ is defined over $\qq$, append it to $S$ and remove it from $T$.

\begin{enumerate}[i.]
\item First check if $f(x(Q)) =0$; if so, $Q$ is a ramification point, and therefore $[Q-\infty]$ is 3-torsion. Remove $Q$ from $T$.

\item Compute $\int_\infty^Q \omega_i$, $i = 1, 2, 3$ on regular basis differentials. If these are $0$ to our $p$-adic precision, then $[Q- \infty]$ is potentially torsion. We verify it is torsion by finding $n \in \zz$ such that $n[Q- \infty] =0 \in J(\qq)$ using Jacobian arithmetic implemented in Magma. If we verify it is torsion, remove $Q$ from $T$.

\item Search in a box for relations with coefficients of small height between $[Q-\infty]$ and known rational divisors on $J(\qq)$. If we succeed, remove $Q$ from $T$. We have explained $Q \in X(\qq_p)_1$ by linearity.

\end{enumerate}

\item Return $S$ and $T$. If
$T=\emptyset$, then $S = X(\qq)$. If $T\not=\emptyset$, then more work must be done to explain
the presence of the remaining points in $T\subseteq X(\qq_p)_1 \setminus S$ and provably determine $X(\qq)$.

\end{enumerate}

\end{alg}

\subsection{Precision Bounds}
\label{prec}

In Algorithm \ref{RatPoint} Step \ref{Integrate} we solve for the zeros of the integral of a regular differential $v$ on $X$. To do this, we solve for the zeros in each residue disk of $\bar{Q} \in X(\FF_p)$, expanding the integral $\int_b^z v = \int_b^Q v + \int_{t(Q)}^{z(t)} v(t) \d t$ for some $Q \in X(\qq_p)$ reducing to $\bar{Q}$, and where $t$ is a uniformizing parameter at $Q$ computed using Algorithm \ref{tinyintegrals}. In practice, we must truncate the $p$-adic power series to be a polynomial with finite $p$-adic precision. Therefore we need to determine how much $p$-adic and $t$-adic precision we need to determine all of the zeros of $\int_b^Q v+\int_{t(Q)}^{z(t)} v(t) \d t$ on the residue disk of $\bar{Q}$.

First we define a {\em normalizing} process for $p$-adic power series.

\begin{alg}
\label{normalize}
Normalize a $p$-adic power series $f$
\\
\noindent Input: $f(t)  \in \qq_p[[t]]$, a $p$-adic power series such that $f\p(t) = \sum_{i\geq 0} a_i t^i \in \zz_p[[t]]$, and $f$ has constant term in $c \in \zz_p$. \\
\noindent Output: $F(x) \in \zz_p[[x]] \setminus p\zz_p[[x]]$ such that the roots of $f(t)$ are in bijection with those of $F$ by the map $r\mapsto pr$.
\begin{enumerate}
	\item Set $F_0(x) \coloneqq f(px)$. Then $F_0(x) = c+ \sum_{i\geq 1} \frac{a_{i-1} p^{i} x^i}{i}$.
	\item Set $\lambda \colonequals \min\left(\left\{ v_p\left(\frac{a_{i-1} p^{i}}{i}\right)\right\}_{i\geq0} \cup \{ v_p(c)\}\right)$.
	\item Return $F(x) \colonequals F_0(x) /p^\lambda$.
\end{enumerate}
\end{alg}

First, it is easy to see $F(x)$ is in $\zz_p[[x]]$. For $i>0$, the coefficient of $x^{i}$ is $\frac{a_{i-1} p^{i-\lambda}}{i}$, which has positive valuation for all $i>0$ (by assumption the constant term has nonnegative valuation).  In
Step 2, we ensure that $F(x) \in \zz_p[[x]] \setminus p \zz_p[[x]]$.
\begin{lem}
\label{finalprec}
Let $f(t) \in \qq_p[[t]]$ be a $p$-adic power series such that $f\p(t) \in \zz_p[[t]]$.  Write $f\p(t) =  a_0 + a_1t + a_2t^2 + \dots $ and $f(t) =c+ a_0 t+ a_1t^2/2 + a_2t^3/3 + \cdots $. Let $F(x)$ be the output of Algorithm~\ref{normalize} on input $f(t)$, so that
\[F(x) = \sum_{i=1}^\infty \frac{a_{i-1} x^{i} p^{i - \lambda} }{i} +c\] where we assume $c \in \zz_p$.

Fix a positive integer $N \in \nn$.
Define
\[
M(N,\lambda) \colonequals \min_m \{m- \lambda -\log_p (m) > N\}
\]
 and let $M=M(N,\lambda)$. Write $F(x) = F_M(x) + F_\infty(x)$ where $F_M(x)$ contains the terms $F$ of degree at most $M$, and $F_\infty(x)$ contains higher order terms.

Then $F_M(x) \equiv F(x) \mod p^{N}$.
\end{lem}

\begin{proof}

Let $b_i \colonequals \frac{a_{i-1} p^{i- \lambda}}{i}$ be the coefficient of $x^{i }$.

To show $F_M(x) \equiv F(x) \mod p^{N}$, we show that for $i \geq M$, $v_p(b_i) > N$. Then $v_p(b_i) \geq v_p (p^{i-\lambda} /i) \geq i - \lambda -\log_p(i)$, and for $i \geq M$, this is greater than $ N$.
\end{proof}

Note that if $c \in \qq_p$ and $v_p(c) <0$, then $F(x)$ and therefore $f(t)$ has no zeros, so we will not need to consider this case.

\begin{rmk}
In order to work with the prime $p = 5 < 2g$, we provide a more detailed analysis of the precision here than in \cite[Section~3.3]{CCExp}. Their algorithm assumes $p > 2g$ in order to obtain explicit formulas for the $p$-adic and $t$-adic precision.
\end{rmk}

\begin{defn}
\label{satHensel}

Let $N$ be a positive integer $G \in \zz_p[[x]]$. Let $\bar{G}$ denote the reduction of $G$ to $\zz/p^N\zz$. We say $G$ \defi{satisfies Hensel's lemma} for $r \in \zz/p^N\zz$ if
\begin{itemize}
\item $\bar{G}(r) = 0$ modulo $p^N$ and
\item $\bar{G}'(r)^2 \neq 0$ modulo $p^N$.
\end{itemize}
\end{defn}

If $G$ satisfies Hensel's lemma for $r$ then by Hensel's Lemma \cite[Theorem~4.1]{Conrad} there exists a unique $\tilde{r} \in \zz_p$ such that $G(\tilde{r}) = 0$ and $|r- \tilde{r}|_p < |\bar{G}'(r)|_p$. Furthermore,
\begin{enumerate}
	\item $|r- \tilde{r}|_p = |\bar{G}(r)|_p/|\bar{G}'(r)|_p$ and \label{digits}
	\item $|\bar{G}'(\tilde{r})|_p = |\bar{G}'(r)|_p$. \label{derivativeval}
\end{enumerate}

\begin{lem}

\label{liftingworks}

Let $F(x)$ be a normalized power series, obtained as output from Algorithm \ref{normalize},
which converges on $\zz_p$. Fix $N$ a positive integer, $M = M(N,\lambda)$, and write $F(x) = F_M(x) +F_\infty(x)$ as in Lemma \ref{finalprec}.

Let $r \in \zz/p^N\zz$ such that $F(x)$ satisfies Hensel's lemma for $r$. Let $\tilde{r}\in\zz_p$ be the unique lift of $r$ satisfying $|r - \tilde{r}|_p < |F_M'(r)|_p.$

Then $\tilde{r} \equiv r \mod p^{N - v_p(F'_M(r))}$.
\end{lem}

\begin{proof}

Since $F$ satisfies Hensel's lemma for $r \in \zz/p^N\zz$, we have $F_M(r) = 0$ modulo $p^N$. Furthermore, by (\ref{digits}), letting $\tilde{r}$ be the unique lift of $r$ satisfying $|F'_M(r)|_p > |r - \tilde{r}|_p$, we have
 \[p^{-N+v_p(F'(r))} \geq |F_M(r)|_p/|F_ M'(r)|_p = |r - \tilde{r}|_p .\]
So $\tilde{r} =r$ modulo $p^{N-v_p(F'_M(r))}$.
\end{proof}

Every root of $F(x)$ reduces to a root of $F_M(x)$. Conversely, suppose we start with a set of roots $R$ in $\zz/p^N\zz$ satisfying Hensel's Lemma for $F_M$, so that for any $r, r\p \in R$, we have
\begin{equation}\label{systemofroots} r \neq r\p \mod p^{N- v_p(F_M'(r))}. \tag{$\dagger$}
\end{equation}
Then Lemma \ref{liftingworks} implies that the roots $r \in R$ lift uniquely to simple roots $\tilde{r}$ of $F$ in $\zz_p$ if $F_M$ satisfies Hensel's lemma for all $r \in R$.

We call a set $R$ with property (\ref{systemofroots}) a \defi{system of roots} for $F_M$ modulo $p^N$. We can compute a system of roots iteratively, lifting roots of $F_M$ modulo $p$ to higher precision, as described in in Algorithm \ref{iterativehensel}.

\begin{alg} Solving for the zeros of a tiny integral in a residue disk
\label{SolveTiny} \\

\noindent Input: $b \in X(\qq)$, $\omega \in \Van(X(\qq))$, $Q \in X(\qq_p)$ with uniformizer $t$ computed as in Algorithm \ref{tinyintegrals}, and a precision $N$. \\
\noindent Output: The zeros in the residue disk of $Q$ of the antiderivative of $\omega$ with constant term $\int_b^Q \omega$, or ``Failure''.

\begin{enumerate}

\item Compute $c=\int_b^Q \omega$ the constant term (if $Q$ is a bad point, $c = \int_b^S \omega + \int_S^Q \omega$ where $S \in X(\qq_p(p^{1/e}))$ is a point near the boundary of the residue disk of $Q$ with $e$ large enough so that the integral from $b$ to $S$ converges).

\item Compute $f\p(t) \colonequals \omega(t)$ to precision $N$ using Algorithm \ref{tinyintegrals}.

\item  Compute $f(t)$ the antiderivative of $f\p(t)$ with constant term $c$, accurate to precision $N\p = N - \delta$, where $\delta$ is the precision loss when integrating discussed in \cite[Section~4]{balatuitman}. If $N\p\leq 0$, return ``Failure''.

\item Compute $M(N\p,\lambda)$ and let $F_{M}(x)$ be as in Lemma \ref{finalprec}.

\item Compute a system of roots $R$ of $F_{M}(x)$ modulo $p^{N\p}$, as described in Algorithm \ref{iterativehensel}.

\item For each $r \in R$, check if $F_M'(r)^2 \neq 0$ modulo $p^{N\p}$.  If this is true for every $r$, these correspond uniquely to the zeros of $f(t)$, as per Lemmas \ref{finalprec} and \ref{liftingworks}. If not, return ``Failure''.

\end{enumerate}

\end{alg}

\begin{rmk}

If $f(t)$ has double roots, then $F_{M}(x)$ always has double roots. In Algorithm \ref{RatPoint}, we solve for common roots of two power series; we check the roots are single roots of at least one of the truncations. In the 1403 Picard curves in our database, at $N = 15$ we found common double roots for a single curve. By using a larger prime,
we were able to provably compute that curve's rational points.

\end{rmk}

\begin{rmk}
According to \cite[Remark~4.2]{balatuitman} the precision $M$ in Balakrishnan and Tuitman's implementation used to truncate $F(x)$ $t$-adically is chosen so that for all $i \geq M$, $v_p(b_i) \geq N$ for $b_i$ the coefficient of $x^{i}$. The chosen $M \geq M(\lambda, N\p)$ is sufficient.
\end{rmk}

\subsection{Beyond rational points}

If we cannot find $P \in X(\qq)$ such that $[P - \infty]$ is infinite order in $J(\qq)$, then we are must instead use a rational divisor on $X$ whose image is infinite order in $J(\qq)$. Because $X$ is a genus $3$ curve with a rational basepoint $\infty \in X(\qq)$, we have a surjection $\Sym^{3}(X) \twoheadrightarrow J$ \cite{MilneJacobian}. Because we assume the rank of $J(\qq)$ is $1$, such a divisor exists, and the points in the support of the divisor are defined over an extension of degree at most $3$.

Using our extension of Balakrishnan and Tuitman's code, it is possible to integrate over any rational divisor of the form $D - \deg(D) b $ for rational primes $p$ split completely in $L$, the smallest number field over which the points in the support of $D$ are defined, and also to points defined over a number field $K$ on a curve $X/\qq$ for $p$ completely split in $K$. The curve $X/\qq$ should be smooth projective and geometrically integral with (not necessarily smooth) affine model $Q(x,y)$ satisfying the same assumptions as in \cite[Assumption~2.6]{balatuitman}. However, we focus our discussion on the situation for superelliptic curves where it is practical to compute rank bounds. Our code is tailored to this case. For non-hyperelliptic superelliptic curves, when computing infinite order points in the Mordell--Weil group $J(\qq)$, $\tt RankBounds$ outputs polynomials\footnote{For a given curve there are often multiple choices of $g$, and we pick the one with the smallest first split prime.} $g(T)$ such that the divisor $\Div(g)$ on $\pp^1$ lifts to a degree $\deg(g(T))$ rational divisor on $X$ under the map $X \to \pp^1$. Then $\deg(g(T))$ is at most the genus of $X$, which is $3$ in our case of Picard curves. (For odd degree hyperelliptic curves $\tt RankBounds$ returns the Mumford representation of the divisor, and one can recover $g$ by taking the polynomial for the $x$-coordinate. Our code does not handle the even degree hyperelliptic case.)

Let $D$ be the degree $d=\deg(g(T))$ rational divisor obtained by lifting $\Div(g)$ to $X$. Then $D = P_1 + \dots + P_d - d \infty$. We use $L$ to denote the splitting field of $g$, which is the smallest number field over which the points $P_j$ are defined. To compute the Chabauty--Coleman set for $X$, we need to compute the set of vanishing differentials, which is done by taking a basis for the kernel of the $1 \times 3$ matrix of integrals of the regular differentials over $D$. The integral $\int_D \omega_i$ is equal to $\sum_{j=1}^d\int_\infty^{P_j}\omega_i$. By choosing a prime $p$ that splits completely in $L$, for any $\fp|p$ in $L$ we have $L_{\fp} \cong \qq_p$. Therefore, we can proceed as before and compute these integrals in $L_{\fp}$.

There are no theoretical innovations needed extend the code to use inert primes $p$ in $L$, but it would require extensive changes to the base code. To integrate using inert primes, one would need to implement the $q$-power Frobenius for residue fields $\ff_q$ where $q$ is a power of $p$. The code directly extending Balakrishnan and Tuitman's code for Coleman integrals and the code for batch computations of rational points and testing the algebricity of points and Jacobian relations are available on GitHub ~\cite{TravisSachiextensions, TravisSachiexperiments}.

\section{The data}

In this section we discuss the overall data and give in-depth explanations on examples of interest. For the rest of the section, we fix $\infty \in X(\qq)$ as our basepoint for the Abel--Jacobi embedding.

We computed the Chabauty--Coleman sets $X(\qq_p)_1$ of $1403$ rank $1$ Picard curves of discriminant bounded by $10^{12}$ and provably determined the set of rational points on these curves. In $859$ cases, we found a rational point $P \in X(\qq)$ whose class $[P - \infty]$ under the Abel--Jacobi map is infinite order. For the other $544$ curves, we used infinite order divisor classes not given by the image of a rational point under the Abel--Jacobi map.\footnote{For the remaining 544 curves, it can be quite computationally expensive to compute $X(\qq)$ depending on the parameters $N$, $e$, and $p$. Individual curves can require up to several hours. This computation was run on a single core of a $28$-core $2.2$ GHz Intel 2 Xeon Gold server with $256$GB RAM.}

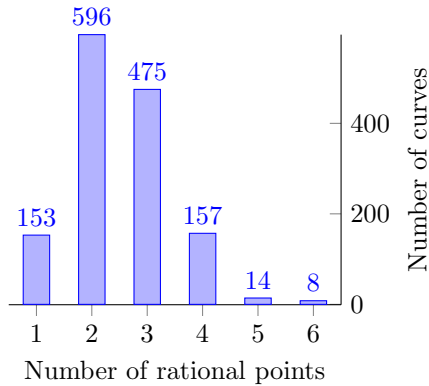
\begin{figure}[!h]
\begin{center}
\scalebox{1}{
\begin{tikzpicture}
  \begin{axis}[
    width = 6cm,
    ybar,
    axis y line*=right,
    axis x line*=bottom,
    enlarge y limits={upper, value=.4},
    enlarge y limits={lower, value=0.015},
    legend style={at={(0.5,-0.15)},
      anchor=north,legend columns=-1},
    ylabel={Number of curves},
    xlabel = {Number of rational points},
    symbolic x coords={1,2,3,4,5,6},
    xtick=data,
    nodes near coords,
    nodes near coords align={vertical},
    ]
\addplot
	coordinates {(1,153) (2,596) (3,475) (4,157) (5,14) (6,8) };
\end{axis}
\end{tikzpicture}
}
\caption{Graph of number of rational points}
\label{figure1}
\end{center}
\end{figure}

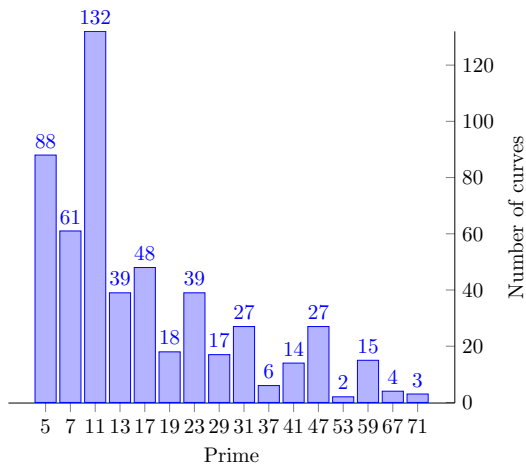
\begin{figure}[!h]
\begin{center}
\scalebox{.80}{
\begin{tikzpicture}
  \begin{axis}[
  	width = 9cm,
    ybar,
    axis y line*=right,
    axis x line*=bottom,
    enlarge y limits={upper, value=.4},
    enlarge y limits={lower, value=0.017},
    legend style={at={(0.5,-0.15)},
      anchor=north,legend columns=-1},
    ylabel={Number of curves},
    xlabel = {Prime},
    symbolic x coords={5, 7, 11, 13, 17, 19, 23, 29, 31, 37, 41, 47, 53, 59, 67, 71},
    xtick=data,
    nodes near coords,
    nodes near coords align={vertical},
    ]
\addplot
	coordinates {(5,88) (7,61) (11,132) (13,39) (17,48) (19,18) (23,39) (29, 17) (31,27)
	(37,6) (41,14) (47,27) (53,2) (59, 15) (67,4) (71,3) };
\end{axis}
\end{tikzpicture}}
\caption{Primes used to determine $X(\qq_p)_1$ for $\deg(g) > 1$}
\label{figure2}
\end{center}
\end{figure}

On all curves, the naive search for points of height at most $1000$ yielded all of the rational points. In Figure \ref{figure1} we display the number of Picard curves with $n$ rational points. Every Picard curve has exactly one point at infinity and it is rational; the curves in our database had at most six rational points. In Figure \ref{figure2} we show the primes used for the curves with divisors not given by the image of a rational point under the Abel--Jacobi map. These are the first split primes in the splitting field of $g$, where the divisor is the rational lift of $\Div(g)$, and $g$ is chosen from a list of possible $g$ returned by $\tt RankBounds$. For ease of reading, we leave out the primes used once: those primes are 43, 61, 83, and 107.

For all but $14$ curves, the only extra points in the Chabauty--Coleman set were the ramification points whose image in $J$ is 3-torsion. For six curves, we found points $P_{\operatorname{tor}} \in X(K)$ such that $[P_{\operatorname{tor}}-\infty] \in J(K)$ is a torsion point of higher order defined over a number field $K$. These torsion points were $9$-torsion and $4$-torsion; the corresponding points $P_{\operatorname{tor}}$ are all of the form $(a,b^{1/3})$ with $a, b \in \zz$.  For five curves, we found extra points $P_{\operatorname{lin}} \in X(K)$ defined over a number field $K$ such that there exists $n \in \zz$ so that $n[P_{\operatorname{lin}}-\infty]$ is infinite order in $J(\qq)$. For three curves, we found an extra point in the Chabauty--Coleman set whose presence is explained by an automorphism of the curve.

\subsection{Examples}
\label{examples}

For some Picard curves $X$ and choices of prime $p$, we found interesting global points in $X(\qq_p)_1$ that were not ramification points or rational points. We discuss four notable examples here.

\begin{eg}

Consider the curve $X: y^3 = x^4 + 6 x^3 - 48 x-64$ with discriminant $1289945088$. Then $J$ has the infinite order point $[(-3, -1) - \infty]$ so we compute the Chabauty--Coleman set at $p = 5$ and find an extra point $T=(t, 1/2 t^2 - 4)$ defined over $\qq [t] / (t^3 - 24 t - 48)$ (LMFDB \href{https://www.lmfdb.org/NumberField/3.1.108.1}{\textsf{3.1.108.1}}). However, if we instead compute the Chabauty--Coleman set at the prime $p = 17$ we find $X(\qq_{17})_1 = X(\qq) \cup W \cup \{ T, S\}$ where $W$ is the pair of $17$-adic ramification points. The point $S$ is a different point  $S = (s, s^2 + 6s + 8)$ defined over a different number field $\qq[s]/(s^3 + 9s^2 + 24 s + 24)$ (LMFDB \href{https://www.lmfdb.org/NumberField/3.1.216.1}{\textsf{3.1.216.1}}).

In the Jacobian of $X$, we have that $18 [T - \infty] = 9[S- \infty]$. Furthermore, relating this back to the known point of infinite order in $J(\qq)$, $[(-3, -1)- \infty]$, we can show that $18[T - \infty ] = 3 [(-3, -1)- \infty]$. So $S$ and $T$ appear in the Chabauty--Coleman set by linearity of the Coleman integral.
\end{eg}

\begin{eg}
The curve $X:y^3 = x^4 + 25x^3 - 78x^2 + 76x - 24$ with discriminant $411210307584$ has Mordell--Weil rank $1$ and possesses the infinite order point $D \in J(\qq)$ that is a lift of $\Div(x^2 - 6x + 4)$.

The first prime that splits completely in $\qq[x]/(x^2 - 6x + 4)$ is $11$, so we compute $X(\qq_{11})_1$. We obtain the rational points, one ramification point, as well as the point $(2, 32^{1/3})$, whose image in $J(\qq)$ under the Abel--Jacobi map is $9$-torsion.
\end{eg}

\begin{eg}
\label{endoexample1}
In their Chabauty--Coleman experiments on genus 3 hyperelliptic curves with Mordell--Weil rank 1, Balakrishnan--Bianchi--Cantoral-Farf\'an--{\c{C}}iperiani--Etropolski \cite{CCExp} find a novel reason for extra points in the Chabauty--Coleman set. They exhibit a curve with two extra points coming from the presence of extra endomorphisms of the Jacobian, which are not explained by linearity or torsion. Motivated by their example, we searched Sutherland's Picard curve database for similar examples and found several. This first example of a curve with an extra point in the Chabauty--Coleman set explained by the automorphism group of the curve comes from searching outside of our rank 1 database.

Let $X: y^3 =x^4-2$ be the Picard curve with discriminant $10319560704000000$. The Jacobian of this curve has rank 2 and has linearly independent non-torsion classes $[(2,3)- \infty]$ and $[(-2,3) - \infty]$ in $J(\qq)$.

The $\overline{\qq}$-automorphism group of $X$ is a group of order $48$ defined over the degree $12$ number field $L \colonequals \qq[x]/(x^{12}-6x^{11}+21x^{10}-50x^9+90x^8-126x^7+191x^6-276x^5-285x^4+950x^3-354x^2-156x+676)$ (LMFDB \href{https://www.lmfdb.org/NumberField/12.0.7652750400000000.2}{\textsf{12.0.7652750400000000.2}}). The group is isomorphic to $\SL_2(\ff_3) \rtimes_{\rho} C_2$ where the image of $\rho: C_2 \to \Aut(\SL_2(\ff_3))$ is given by conjugation by an element of order $2$. Let $p = 13$ be the first prime that splits completely in $L$. Computing the Chabauty--Coleman set at $p$, we find $X(\qq_{13})_1 = X(\qq) \cup W \cup T \cup A$ where $W$ is the set of ramification points, $T$ is the set of points whose image in $J(\qq)$ is torsion, and $A$ is a set of points that will be explained only by automorphisms of the curve, as shown in Table $\ref{pointsinXqp1}$.

\begin{table}
\begin{center}
  \begin{tabular}{ l p{0.8\textwidth} }
    \toprule
    Set & Points \\ \midrule
    $X(\qq) $ &  $(2,3), (-2,3), \infty$\\
    $W$  & Four Galois conjugates of $((5/2)^{1/4}, 0)$  \\
    $T$ & Twelve Galois conjugates of $((45/2)^{1/4}, 40^{1/3})$, Three Galois conjugates of $(0, (-5)^{1/3})$\\
    $A$ & Two Galois conjugates of $((-4)^{1/2}, 3)$ \\
    \bottomrule
  \end{tabular}
  \end{center}
  \caption{Points in $X(\qq_{13})_1$}
  \label{pointsinXqp1}
  \vspace{-.3in}
\end{table}
To verify that the points in $T$ give rise to torsion classes in the Jacobian, while those in $A$ do not, we compute the Coleman integrals $\int_\infty^P \omega_i$, $i = 1, 2, 3$ on regular differentials for each $P$ in $T$ and $A$ using the new functionality for integrating to points defined over number fields. These integrals are $0$ on points in $T$ to our $p$-adic precision, but not zero on points in $A$. We then algebraically check that for $P \in T$, the class $n[P - \infty]=0 \in J(\qq)$ for some $n \in \zz$ using Jacobian arithmetic in Magma. The image of each point in $T$ of the form $((45/2)^{1/4}, 40^{1/3})$ under the Abel--Jacobi map is a $12$-torsion point on the Jacobian, while those of the form $(0, (-5)^{1/3})$ are $4$-torsion points in $J$.

On the other hand, we show that the classes $A_1 \coloneqq[((-4)^{1/2}, 3)- \infty] $ and $A_2 \coloneqq [(-(-4)^{1/2}, 3)- \infty]$ do not have $\zz$-linear relations with $P_1 \coloneqq [(2,3) - \infty]$ and $P_2 \coloneqq [(-2,3) - \infty]$ in $J(\qq)$. Suppose $n A_1 = a P_1 + b P_2$ for some $n,a,b \in \zz$ with $n \neq 0$. Then $nA_1^c = a P_1 + bP_2$ where $c$ is complex conjugation. But $A_1^c = A_2$, and subtracting we would have $n(A_1 - A_2) = 0$ in $J(\qq)$. However, computing the Coleman integrals to $A_1$ minus the Coleman integrals to $A_2$ we get a nonzero value, so $A_1 - A_2$ cannot be torsion.

Therefore we need a new reason to explain the presence of the points in $A$ in the Chabauty--Coleman set. This reason will come from examining the automorphisms of $X$. When the vanishing differential is an eigenvector for a nontrivial automorphism of the curve, there is potential for extra points to appear in the Chabauty--Coleman set. Recall that $\omega_1 = y\d x/f, \omega_2 = xy\d x/f,$ and $\omega_3 = y^2 \d x/f$. Note $X$ has an order $4$ automorphism given by $\varphi: x \mapsto i x$ that acts by pullback on the differentials: $\varphi^* \omega_1 = i \omega_1, \varphi^* \omega_2 = - \omega_2, \varphi^* \omega_3 = i\omega_3$. The set of vanishing differentials $\Van(X(\qq))$ has dimension 1; let $v  \in \Van(X(\qq))$ be a nonzero element. We will show $v$ is an eigenvector for $\varphi^*$ with eigenvalue $i$.

First we will explain this phenomenon by showing why the points of $A$ are in the Chabauty--Coleman set, assuming $v$ is an eigenvector for $\varphi^*$ with eigenvalue $i$. It is enough to show that integrating $A_j$ against the vanishing differential $v$ is zero. Note that choosing square roots, $\varphi((2,3)) =((-4)^{1/2}, 3)$ and $\varphi((-2,3)) = (-(-4)^{1/2}, 3)$. Using change of variables and applying the fact that our vanishing differential is an eigenvector, we have
\[ \int_{A_1} v = \int_{\varphi(\infty)}^{\varphi((2,3))} v = \int_{P_1} \varphi^* v = i \int_{P_1} v =0\] and similarly for $A_2$.  Thus the points of $A$ are in the Chabauty--Coleman set because of $\varphi$.

Now we prove that $v$ is an eigenvector for $\varphi^*$. Write $v = A \omega_1 + B \omega_2 + C \omega_3$, with $A, B, C \in \qq_p$. The vanishing differential is defined (up to a scalar multiple) by the equation $\int_{P_j} v= 0$ for $j = 1, 2$. Acting by $\varphi^2$ we have \[0=\int_{P_2} v=\int_{\varphi^2(\infty)}^{\varphi^2((2,3))} v = \int_{P_1}(\varphi^*)^2 v = \int_{P_1} -A \omega_1 +B \omega_2 -C\omega_3 .\]
Adding this to the fact that $\int_{P_1} v = \int_{P_1} A\omega_1 +B \omega_2 +C \omega_3 =0$ we see that
\[2 \int_{P_1} B \omega_2 = 0.\]
We compute that the integral over $P_1$ of $\omega_2$ is not zero, so $B$ must be zero, and $v = A \omega_1 + C \omega_3$ is indeed an eigenvector for $\varphi^*$ with eigenvalue $i$.

\end{eg}

\begin{eg}
In addition to the previous example, we found three examples of extra points coming from automorphisms of the curve that arose during the computation of the Chabauty--Coleman sets in our rank 1 database. We describe one here. Let $X: y^3 = x^4 + 2x^3 + 6x^2 + 5x + 2$ be the Picard curve with discriminant $277826509467$ and Mordell--Weil rank $1$. There are no small affine rational points on this curve, so we use an infinite order point $D \in J(\qq)$ that is a lift of $\Div(x^2 + x - 1)$, and we compute the Chabauty--Coleman set for $X$ at $p = 11$, the first split prime in $\qq[x]/(x^2+ x -1)$. We obtain: \[X(\qq_{11})_1 = \{ \infty, (-1/2, \sqrt[3]{13/16})\}.\] Let $R \coloneqq(-1/2, \sqrt[3]{13/16})$. To check if $R$ is torsion, we compute the Coleman integrals on the regular basis differentials $\int_\infty^R \omega_i$, $i=1, 2, 3$, which are not all zero. The presence of $R$ in $X(\qq_{11})_1$ also cannot be explained by relations in $J(\qq)$: suppose $n[ R - \infty] = mD$ for some $n, m \in \zz \setminus \{ 0\}$. Let $\zeta_3$ be a primitive third root of unity. Acting pointwise by $\sigma: (x,y) \mapsto (x,\zeta_3 y)$ on each divisor, we have $nR^\sigma - n\infty^\sigma = mD^\sigma$. However, $mD^\sigma \in J(\qq)$, while the left hand side is not.

By computing the automorphism group of $X$ over $\qq[x]/(16x^3 - 13)$, we find $X$ has an order $2$ automorphism $\varphi$ sending $x \mapsto -x -1 $ and fixing $y$. The pullback $\varphi^*$ acts on Tuitman's basis of regular differentials by $\varphi^* \omega_1 = - \omega_1$, $\varphi^* \omega_2 = \omega_2$, and $\varphi^* \omega_3 = -\omega_3$. By a similar method used in Example \ref{endoexample1}, comparing the Coleman integral over $D$ against $v$, and the Coleman integral over $D$ against $\varphi^*v$, we can show that the vanishing differentials are of the form $v = A\omega_1 + C\omega_3$, and are eigenvectors for $\varphi^*$ with eigenvalue $-1$. Furthermore, $\varphi(R) = R$ and $\varphi(\infty) = \infty$ and therefore \[\int_\infty^R v = \int_{\varphi(\infty)}^{\varphi(R)} v = \int_\infty^R \varphi^* v = -\int_\infty^R v.\] So $2 \int_\infty^R v = 0$, showing $R \in X(\qq_{11})_1$.

\end{eg}

\section{Acknowledgments} We are very grateful to Jennifer Balakrishnan for suggesting this project and for numerous helpful discussions and suggestions. We are also grateful to Drew Sutherland for providing us with data, and to Alex Best and Borys Kadets for their generous help and advice. We are also very thankful to the anonymous referees for their helpful suggestions for improving the paper.

\appendix
\section{Precision heuristics}
\label{eheuristics}

Recall that, in the process of computing (non-tiny) Coleman integrals $\int_P^Q \omega$ where one endpoint $Q$ is in a bad disk, we must break the integral up into $\int_P^S\omega + \int_S^Q \omega$ such that $S$ is near the boundary of the disk. In other words, we find a point $S$ defined over a totally ramified extension $\qq_p(p^{1/e})$. Balakrishnan and Tuitman's code takes as an input this $e$, but it can be a balancing act to decide what value to enter: too small and the integral will not converge, but the larger $e$ is, the longer the runtime.

\begin{table}[hb]
\begin{center}
  \begin{tabularx}{315pt}{ l l l l l } \toprule
  Discriminant : Curve  & $5$  & $7$ & $11$ & $13$ \\
  \midrule
    $31492800:y^3=x^4 + 3x^3 - 3x + 1 $& bad   &$50$ & $85$ & $100$ \\
$47258883:y^3=x^4 + 2x^3 - x^2 - x$& $40$ & bad  & $85$ & $100$\\
$70858800:y^3=x^4 + 3x^3 - x^2 - 4x + 2$& bad &$ 50$& $85$ & $100$\\
$106288200:y^3=x^4 + x^3 - 66x^2 - 324x - 432$ & bad   & $30$& $50$& $55$ \\
$151322904:y^3=x^4 + 11x^3 - 32x^2 + 28x - 8$&  $30$  & $35$ & $50$ & $100$ \\
$212891328:y^3=x^4 + 5x^3 + 4x^2 - 5x + 1$& $10$ & $50$ & $85$ & bad\\
$105303439827:y^3=x^4 + 4x^3 + x^2 - 3x - 1$&$30$ & $50$ & $85$ & $95$ \\
$108095630841:y^3=x^4 + 4x^3 + 6x^2 - 9x$& $35$ & $40$ & bad & $55$\\
$113232992256:y^3=x^4 + x^3 - 4x^2 - 8x$& $30$ & $50$ & $80$ & $100$\\
$988868881152:y^3=x^4 + x^2 - 2x + 3$& $10$& $55$& $85$ & $100$\\
  			\bottomrule
  \end{tabularx}
  \caption{Table for $p = 5,7,11,13$, $N = 10$}
  \end{center}
\end{table}

\begin{table}
\begin{center}
  \begin{tabularx}{360pt}{ l l l l l l l l l }
      \toprule
  Curve & $23 $& $31$&$41$&$53$&$61$ &$71$&$83$&$97$ \\
      \midrule
    $y^3=x^4 + 3x^3 - 3x + 1$ &$180$ & $245$& $325$ &$420$&$485$&$565$&$660$& $770$\\
$y^3=x^4 + 2x^3 - x^2 - x$ & $180$ & $245$& $325$ &$420$&$485$&$565$&$660$& $770$\\
$y^3=x^4 + 3x^3 - x^2 - 4x + 2$&$180$ & $245$& $325$ &$420$&$485$&$565$&$660$& $770$\\
    $y^3=x^4 + x^3 - 66x^2 - 324x - 432$ & $95$ & $105$ & $155$ &$200$ & $205$ & $240$ &$280$&$325$ \\
    $y^3=x^4 + 11x^3 - 32x^2 + 28x - 8$ &$180$ & bad & $325$ &$420$&$485$&$565$&$660$& $770$\\
    $y^3=x^4 + 5x^3 + 4x^2 - 5x + 1$ &$180$ & $240$& $325$ &$420$&$485$&$565$&$660$& $770$\\
    $y^3=x^4 + 4x^3 + x^2 - 3x - 1$ &$180$ & $245$& $325$ &$420$&$485$&$565$&$660$& $770$\\
    $y^3=x^4 + 4x^3 + 6x^2 - 9x$  & $95$ & $120$ & bad  & $200$ & $230$ & $270$ & $315$ & $365$ \\
    $y^3=x^4 + x^3 - 4x^2 - 8x$ & $180$ & $245$ & $320$ &bad &$485$ &$565$ &$660$&$770$\\
    $y^3=x^4 + x^2 - 2x + 3$& $180$ &$245$ & $325$ &$420$& $485$ &$565$  &$660$  &$770$\\
    \bottomrule
  \end{tabularx}
\caption{Table for select $p\geq 23$, $N=10$}
\end{center}
\end{table}

This section gives tables of the smallest $e$ needed for $\tt effective\_chabauty$ rounded up to the nearest multiple of five, fixing $p$ and $N$, for $10$ arbitrarily chosen curves of rank $1$ with a rational point whose class in the Jacobian is of infinite order taken from Sutherland's database. Entries ``bad'' denote that a curve has bad reduction at that prime. For the tables, after small primes $p \leq 13$ we chose a sample of primes less than $100$ spaced roughly by $10$.

For every curve $X$ in the full database of $1403$ Picard curves, we also recorded the $e$ value used to compute $X(\qq_p)_1$ and to select a non-torsion divisor when $N = 15$. We initialized $e$ at $40$ and then incremented by $20$ until the integrals to compute $X(\qq_p)_1$ converged. These values can be found in the master list of Chabauty data on the 1403 Picard curves \texttt{masteralldata.txt} available in the repository \cite{TravisSachiexperiments}.

\section{Iterative Hensel lifting}
In this section we give an algorithm for computing approximate roots in $\zz_p$
of a polynomial with coefficients in $\zz_p$. Concretely, the elements of $\zz_p$
are represented by integers in $[0,p^N-1]$ for some positive integer $N$, the coefficient precision (here, $[a,b]$ is the set of all integers $k$ such that $a\leq k \leq b$).
Let $f\in \zz[x]$ be a polynomial whose coefficients are in $[0,p^N-1]$. We compute elements $r\in [0,p^N-1]$ such that
\begin{enumerate}
\item $f(r) = 0 \pmod{p^N}$;
\item there exists an integer $0\leq k_r\leq N$ such that $r\in [0,p^{k_r}-1]$ and
$f(r+p^{k_r}s)$ is identically zero in $\zz/p^N[s]$; \label{bestdigits}
\item if $r'\in[0,p^{k_r-1}-1]$ satisfies $r'=r\pmod{p^{k_r-1}}$ then $f(r'+p^{k_r-1}s)$ is not
identically zero in $\zz/p^N[s]$; and
\item if $2v_p(f'(r))<N$ then there exists $\tilde{r}\in \zz_p$ such that $f(\tilde{r})=0$ and
$|\tilde{r}-r|\leq p^{-k_r}$ (in other words, the approximation of $\tilde{r}$ by $r$
is correct to precision $k_r$). \label{Henselproperty}
\end{enumerate}
These properties follow in directly from Algorithm \ref{iterativehensel}, except (\ref{Henselproperty}) which is an application of Hensel's lemma (see Definition \ref{satHensel}).

From these properties, we see the algorithm computes a system of roots as described in Section~\ref{prec}, with the additional property (\ref{bestdigits}) which keeps track of the precision of the roots.  To
compute a system of roots, we inductively lift roots of $f$ mod $p^i$ to roots mod $p^{i+1}$ in what is
essentially a depth-first-search through the elements of $\zz/p^N$.
\begin{alg}\label{iterativehensel} Solve for a system of roots $R$ of $f \in \zz/p^N \zz[x]$. $\left.\right.$

\noindent Input: Prime $p$, precision $N$, and $f\in \zz/p^N\zz[x]$ such that $f \mod p$ is not
identically zero. \\
\noindent Output: A set $R$ of roots in $\zz/p^N\zz$ such that if $r, r\p \in R$, then $r \neq r\p \mod p^{N - v_p(f'(r))}$.
\\
\begin{enumerate}
\item Let $A = []$.
\item Let $B = \{a \in \ff_p: f(a) = 0 \mod p\}$ be the $\ff_p$-roots of $f$.
\item Update $A = [(b,1) : b \in B]$.
\item Let $Z = []$.
\item While $A$ is nonempty:
\begin{enumerate}
\item Let $(a,i)$ be the first element of $A$. Remove $(a,i)$ from $A$.
\item Let  $g_0(s) \colonequals f(a+p^is)$.
\item If $g_0(s) \neq 0 \pmod{p^N}$:
\begin{enumerate}
\item Let $v$ be the minimum valuation of the coefficients of $g_0$ and set
$g\colonequals g_0/p^v$.
\item Let $B$ be the roots of $g$ in $\ff_p$ and prepend $[(b,i+1) : b \in B]$
to $A$.
\end{enumerate}
\item Else:
\begin{enumerate}
\item Append $(a,i)$ to $Z$.
\end{enumerate}
\end{enumerate}
\item Return $Z$.
\end{enumerate}

\end{alg}

\bibliographystyle{alpha}

\newcommand{\etalchar}[1]{$^{#1}$}

\end{document}